\documentclass[12pt]{amsart}

\usepackage[colorlinks]{hyperref}
\hypersetup{
citecolor = {blue}
}
\usepackage{color,graphicx,shortvrb}
\usepackage[latin 1]{inputenc}
\usepackage{amssymb}
\usepackage[active]{srcltx} 

\usepackage{enumerate}

\newtheorem{theorem}{Theorem}[section]

\newtheorem{corollary}[theorem]{Corollary}

\newtheorem{definition}[theorem]{Definition}
\newtheorem{lemma}[theorem]{Lemma}

\newtheorem{proposition}[theorem]{Proposition}
\newtheorem{remark}[theorem]{Remark}

\newtheorem{example}[theorem]{Example}

\newtheorem*{theorem-non}{Theorem}

\def\RR{{\mathbb{R}}}

\def\11{\textbf{$1$}}

\def\io{\int_\Omega }
\def\ido{\int_{\partial \Omega}}
\def\d{\hbox{\rm div\,}}
\def\z{{\bf z}}

\DeclareGraphicsExtensions{.jpg,.pdf,.png,.eps}

\usepackage{enumerate}

\begin{document}


\title[$1$--Laplacian with subcritical source term]{Elliptic equations involving the $1$--Laplacian and a subcritical source term}

\thanks{The first author is partially supported by  MINECO-FEDER
Grant MTM2015-68210-P (Spain), MINECO Grant BES-2013-066595 (Spain) and Junta de Andaluc\'{\i}a FQM-116 (Spain). The second author is supported by MINECO-FEDER under grant MTM2015--70227--P}

\author[A. Molino Salas and S. Segura de Le\'on]{Alexis Molino Salas and Sergio Segura de Le\'on}

\address{Alexis Molino Salas
\hfill \break\indent Departamento de An\'alisis Matem\'atico, Universidad de Granada,
\hfill\break\indent Avenida Fuentenueva S/N,18071 Granada, Spain} \email{{\tt amolino@ugr.es }}

\address{ Sergio Segura de Le\'on
\hfill\break\indent Departament d'An\`alisi Matem\`atica, Universitat de Val\`encia,
\hfill\break\indent Dr. Moliner 50,
46100 Burjassot, Valencia, Spain} \email{{\tt sergio.segura@uv.es.}}

\date{}

\begin{abstract}
In this paper we deal with a Dirichlet problem for an elliptic equation involving the $1$--Laplacian operator and a source term. We prove that, when the growth of the source is subcritical, there exist two bounded nontrivial solutions to our problem. Moreover,  a Poho\u{z}aev type identity is proved, which holds even when the growth is supercritical. We also show explicit examples of our results.
\end{abstract}

\maketitle

 \thispagestyle{empty}

\section{Introduction}
This paper is concerned to the following Dirichlet problem for the $1$--Laplacian operator and a subcritical source term, whose model problem is
\begin{equation}\label{subcritical}
\left\{
\begin{array}{cc}
-{\rm{div}}\left(\displaystyle \frac{Du}{|Du|} \right)=|u|^{q-1}u,& \hbox{ in } \Omega,
\\
\\
u=0 & \hbox{ on } \partial \Omega,
\end{array}
\right.
\end{equation}
where $\Omega\subset \RR^N$ ($N\geq 2$) is an open bounded set with Lipschitz boundary and $0<q<\frac1{N-1}$. Our aim is to obtain nontrivial solutions (in the sense of Definition \ref{definition}) and study their properties.

We point out that similar problems have many applications and have been studied for a long time. Indeed, the study of steady states of reaction--diffusion equations have systematically been studied since the late 1970s (see \cite{F} and \cite{Ni} for a more recent survey). More precisely, Dirichlet problems with $p$--Laplacian type operator ($p>1$) having a term with a subcritical growth, that is:
\begin{equation}\label{p-subcritical}
\left\{
\begin{array}{lc}
 -\Delta_pu=|u|^{q-1}u, & \hbox{ in } \Omega,
 \\[5mm]
 u=0, & \hbox{ on } \partial \Omega,
 \end{array}
 \right.
 \end{equation}
 with $0<q<p^*-1$ (where $p^*$ stands for the Sobolev conjugate), have extensively been considered in the theory of Partial Differential Equations by using different approaches (for a background we refer to \cite{AA} and \cite{Mawhin}). For instance in \cite{DJM} the authors, by using the well--known ``Mountain Pass Theorem" by Ambrosetti and Rabinowitz \cite{AR}, firstly proved that the trivial solution is a local minimum of the corresponding energy functional and then, since the functional has a mountain pass geometry, they find other critical points (one positive and another one negative), which obviously are solutions to problem \eqref{p-subcritical} . We point out that the proof of the Palais--Smale condition relies on the reflexivity of the energy space $W_0^{1,p}(\Omega)$. Moreover, the restriction $q<p^*-1$ ensures that the imbedding $W_0^{1,p}(\Omega) \hookrightarrow L^q(\Omega)$ is compact, being this fact  essential for the approach used in \cite{DJM}.

 The 1--Laplace operator appearing in \eqref{subcritical} introduces some extra difficulties and special features. We recall that in recent years there have been many works devoted to this operator (we refer to the pioneering works \cite{K, K2, D1, ABCM} and the related papers \cite{ACDM, ACM1, BCN, CT, D2, D3}). One of the main interests for studying the Dirichlet problem for equations involving the 1--Laplacian comes from the variational approach to image restoration (we refer to \cite{ACM} for a review on the first variational models in image processing and their connection with the $1$--Laplacian). This has led to a great amount of papers dealing with problems that involve the 1--Laplacian operator. In spite of this situation, up to our knowledge, this is the first attempt to analyze problem \eqref{subcritical}.

 The natural energy space to study problems involving the 1--Laplacian is the space $BV(\Omega)$ of functions of bounded variation, i.e., those $L^1$--functions such that their distributional gradient is a Radon measure having finite total variation.
 In order to deal with the $1$--Laplacian operator, a first difficulty occurs by defining the quotient $\displaystyle \frac{Du}{|Du|}$, being $Du$ just a Radon measure. It can be overcome through the theory of pairings of $L^\infty$--divergence--measure vector fields and the gradient of a BV--function (see \cite{Anze}). Using this theory, we may consider a vector field $\z\in L^\infty(\Omega;\RR^N)$ such that $\|\z\|_\infty\le 1$ and $(\z, Du)=|Du|$, so that $\z$ plays the role of the above ratio.
  In general, the Dirichlet boundary condition is not achieved in the usual trace form, so that a very weak formulation must be introduced: $[\z, \nu]\in{\rm sign\,}(-u)$, where $[\z,\nu]$ stands for the weak trace on $\partial\Omega$ of the normal component of $\z$.

  We point out that the space $BV(\Omega)$ is not reflexive, so that we cannot follow the arguments of \cite{DJM}. Instead, we apply the results in \cite{DJM} for problem \eqref{p-subcritical} getting nontrivial solutions $w_p$ and then we let $p$ goes to 1. Hence, one of our biggest concerns will be that constants appearing in the proof do not depend on $p$. The other major difficulty we have to overcome  is to check that the limit function $w=\lim_{p\to1}w_p$ is not trivial.

\subsection{Assumptions and main result}

Let us state our problem and assumptions more precisely. We consider the general problem
\begin{equation}\label{problem}
\left\{
\begin{array}{cc}
-\textrm{div}\left(\displaystyle \frac{Du}{|Du|} \right)=f(x,u),& \hbox{ in } \Omega,
\\
\\
u=0, & \hbox{ on } \partial \Omega.
\end{array}
\right.
 \tag{$P$}
\end{equation}
Here, the source term  $f:\Omega \times \RR \to \RR$ is a Carath\'eodory function satisfying the following hypotheses
\begin{enumerate}
\item[(i)] There exists $\alpha >0$ such that
\[
\lim_{s \to 0} \sup \frac{|f(x,s)|}{|s|^\alpha}<\infty, \qquad \hbox{ uniformly in } x\in \Omega.
\]
\item[(ii)] There exist $q\in \left(0, \frac{1}{N-1} \right)$ and $C>0$ such that
\[
|f(x,s)|\leq C \left(1+|s|^q  \right), \qquad x\in \Omega, s\in \RR.
\]
\item[(iii)] There exist $\kappa >1$ and $s_0> 0$ such that
\[
0<\kappa F(x,s)\leq sf(x,s), \qquad x\in \Omega, |s|\geq s_0,
\]
\end{enumerate}
where $F(x,s)=\int_0^s f(x,t)dt$. We deal  with solutions of problem \eqref{problem} in the sense of Definition \ref{definition} (see next section). Our main result is stated as follows:

\begin{theorem}\label{teo1} Under the above assumptions, there exist at least two nontrivial solutions $v,w \in BV(\Omega)\cap L^\infty(\Omega)$ of problem \eqref{problem}. Moreover, $v\leq 0 \leq w$ a.e. $x \in \Omega$.
\end{theorem}

The proof of existence considers approximating $p$--Laplacian pro\-blems and then the limit as $p\to 1^+$ of their nontrivial solutions $w_p\,$ is taken. To this end, it is essential to achieve the existence of a positive constant $\tilde C$ independent of $p$ such that
\begin{equation}\label{karak}
\|w_p\|_{W_0^{1,1}(\Omega)}\leq \tilde C\,,
\end{equation}
 so that they are uniformly bounded in $W_0^{1,1}(\Omega)$. However, we carefully have to check that their limit is not the trivial solution.

  As far as the regularity of solutions is concerned, we further prove that they are bounded.
To prove the boundedness of the solutions a crucial point is the estimate \eqref{karak}. We would like to highlight that the usual Stampacchia truncation method with $p-$Laplacian problem does not work here since the problem becomes superlineal when $p$ tends to $1$ (i.e. $p-1<q$).

  Finally, in Proposition \ref{Pohozaev} we state  a Poho\u{z}aev type identity for solutions belonging to $W^{1,1}(\Omega)$. The important point to note here is, unlike $p-$Laplacian problems, the existence of solutions for any growth conditions of the source term. This is confirmed by dealing with explicit examples in the ball.

This paper is organized as follows: in the next section on Preliminaries we introduce the space of functions of bounded variation and we give some definitions and properties of Anzellotti's theory. In addition, we raise the problem \eqref{problem} in a variational framework. Section 3 is devoted to the proof of existence and regularity of nontrivial solutions. To finish, in Section 4 a Poho\u{z}aev type identity is obtained. For the sake of completeness, we include there some examples.

\section{Preliminaries}
Throughout this paper, the symbol $\mathcal H^{N-1}(E)$ stands for the $(N - 1)$--dimensional
Hausdorff measure of a set $E\subset\RR^N$ and $|E|$ for its
Lebesgue measure. Moreover, $\Omega\subset \RR^N$ denotes an open bounded set with Lipschitz boundary. Thus, an outward normal unit
  vector $\nu(x)$ is defined for $\mathcal H^{N-1}$--almost every
  $x\in\partial\Omega$.

   We will denote by $W^{1,q}_{0}(\Omega)$ the usual Sobolev space, of measurable functions having weak gradient in $L^{q}(\Omega;\RR^N)$ and zero trace on $\partial \Omega$. Finally, if $1\leq p< N$, we will denote by $\displaystyle p^{*}=Np/(N-p)$ its Sobolev conjugate exponent.
   Furthermore,  $BV(\Omega)$ will denote the space of functions of bounded variation:
\[
BV(\Omega)=\left\{u\in L^1(\Omega)\,:\, Du \hbox{ is a bounded Radon measure   }   \right\}
\]
where $Du:\Omega \to \RR^N$ denotes the distributional gradient of $u$.
In what follows, we denote the distributional gradient by $\nabla u$ if it belongs to $L^1(\Omega;\RR^N)$. We recall that the space $BV(\Omega)$  with norm
\[
\|u\|_{BV(\Omega)}=\io |Du| +\io |u|
\]
is  a Banach space which is non reflexive and non separable.

On the other hand, the notion of a trace on the boundary can be extended to functions $u\in BV(\Omega)$, so that we may write $u\big|_{\partial\Omega}$, through a bounded operator $BV(\Omega)\hookrightarrow L^1(\partial\Omega)$, which is also onto. As a consequence, an equivalent norm on $BV(\Omega)$ can be defined (see \cite{AFP}):
\begin{equation*}
\|u\|=\io |Du| + \ido |u|\, d\mathcal{H}^{N-1},
\end{equation*}
where $\mathcal{H}^{N-1}$ denotes the $(N-1)$--dimensional Hausdorff measure. We will often use this norm in what follows.
In addition, the following continuous embeddings hold
\[
BV(\Omega) \hookrightarrow L^{m}(\Omega)\,,\quad\hbox{for every }1\le m\le\frac{N}{N-1}\,,
\]
which are compact for $1\leq m <\frac{N}{N-1}$.

  Moreover, we will use some functionals  which are lower semicontinuous with respect to the $L^1$--convergence. Besides the BV--norm, we also apply the lower semicontinuity of the functional given by
\[
u \mapsto \int_{\Omega} \varphi \, |Du|,
\]
where $\varphi$ is a nonnegative smooth function. For further properties of functions of bounded variations, we refer to \cite{AFP}

\medskip

Since our concept of solution lies on the Anzellotti theory, we next introduce it.
Consider $X_N(\Omega)=\left\{{\bf z}\in L^\infty(\Omega;\RR^N) \, : \, \d \,{\bf z}\in L^N(\Omega)\right \}$. For $\z\in X_N(\Omega)$ and $u\in BV(\Omega)$ we denote by
$
({\bf z},Du):\mathcal{C}_c^\infty(\Omega)\to \RR
$
the distribution introduced by Anzellotti (\cite{Anze}):
\begin{equation}\label{green-anze}
\left<({\bf z}, Du), \varphi \right>=-\io u\, \varphi \, \d \, {\bf z}-\io u\, {\bf z} \, \nabla \varphi, \quad \forall \, \varphi \in \mathcal{C}_c^\infty(\Omega) \,.
\end{equation}
Moreover, in \cite{Anze} (see also \cite[Corollary C.7, C.16]{ACM}) it is proved that $({\bf z}, Du)$ is a Radon measure with finite total variation and for every Borel $B$ set with $B\subseteq U\subseteq \Omega$ ($U$ open) it holds
\begin{equation}\label{Borel}
\left| \int_B ({\bf z}, Du) \right| \leq \int_B \left| ({\bf z}, Du)  \right | \leq \|\z\|_{L^\infty(U)}\int_B |Du|\,.
\end{equation}
We recall the notion of weak trace on $\partial \Omega$ of the normal component of ${\bf z}$ defined in \cite{Anze} as the application $\left[{\bf z}, \nu \right]:\partial \Omega \to \RR$, being $\nu$ the outer normal unitary vector of $\partial \Omega$, such that $\left[\z, \nu \right]\in L^\infty(\partial \Omega)$ and $\|\left[{\bf z}, \nu \right]\|_{L^\infty(\partial \Omega)} \leq \| \z \|_{L^\infty(\Omega;\RR^N)}$. Furthermore, this definition coincides with the classical one, that is,
\begin{equation}\label{strip}
\left[{\bf z}, \nu \right]=\z \cdot \nu,\quad \hbox{for }\, \z \in \mathcal{C}^1(\overline \Omega_\delta; \RR^N)\,,
\end{equation}
where $\Omega_\delta=\left\{x\in \Omega \,:\, {\rm dist}(x,\partial \Omega)<\delta   \right\}$, for some $\delta>0$ sufficiently small.
In \cite{Anze} a Green formula involving the measure $\left(\z, Du   \right)$ and the weak trace $\left[{\bf z}, \nu \right]$ is established, namely:
\begin{equation}\label{Green-2}
\io \left(\z,Du   \right)+ \io u\, \d \, \z =\ido u \left[\z,\nu  \right]d\mathcal{H}^{N-1}
\end{equation}
being $\z \in X_N(\Omega)$ and $u \in BV(\Omega)$.

  Next, we give the definition of  solution to our problem
\begin{definition}\label{definition}
We say that $u\in BV(\Omega)$ is a solution of problem \eqref{problem} if there exists a vector field $\z\in L^\infty(\Omega;\RR^N)$ with $\|\z\|_\infty\leq 1$ and such that
\begin{enumerate}
\item[(1)] $- { \rm{ div}} \, \z = f(x,u)$ in $\mathcal{D}^\prime (\Omega)$,
\item[(2)] $\left(\z, Du\right)=|Du|$ as measures on $\Omega$,
\item[(3)] $\left[\z, \nu   \right ]\in \rm{sign}(-u)$ on $\partial \Omega$.
\end{enumerate}
\end{definition}

\begin{remark}\rm
We remark that our solution belongs to $BV(\Omega)\subset L^{\frac N{N-1}}(\Omega)$. Thus condition (ii) satisfied by function $f$ leads to
\[
|f(x,u(x))|\le C \left(1+|u(x)|^q  \right)\in L^{\frac N{q(N-1)}}(\Omega)
\]
for certain $1<q<\frac1{N-1}$, wherewith $f(\cdot,u)\in L^{N}(\Omega)$. It follows from (1) in the above definition that $\d\z\in L^{N}(\Omega)$, so that the Anzellotti theory is available.
\end{remark}

\begin{remark}\rm
In principle, condition (1) in Definition \ref{definition} only allows us to take test functions in the space $\mathcal{C}_c^\infty(\Omega)$. We explicitly point out that, as a consequence of the Anzellotti theory, we may choose any $w\in BV(\Omega)$ as a test function. Then, Green's formula \eqref{Green-2} implies
\[
\io (\z, Dw)-\io f(x, u)w=\ido w[\z,\nu]\, d\mathcal H^{N-1}\,.
\]
\end{remark}

Observe that the vector field $\z$ need not be unique. For instance, we may choose $\z = (1,0,\cdots,0)$ or $\z = (0,1,\cdots,0)$ to check that $u\equiv 0$ is solution of \eqref{subcritical}.

\medskip

In order to introduce a variational setting of problem \eqref{problem}  we recall the notion of subdifferential of a convex operator.
\begin{definition}
Let $H:BV(\Omega) \to \RR$ be a convex operator. For every $u\in BV(\Omega)$ we denote by $\partial H(u)$, the subdifferential of $H$ in $u$, as the set
\[
\left\{\xi \in BV(\Omega)^\prime\,:\, H(u)+\xi (v-u)\leq H(v),\hbox{ for all } v\in BV(\Omega)\right\}
\]
\end{definition}
\begin{remark}\rm
Using this definition it is easy to check that $u_0$ is a global minimum of $H$ if and only if $0\in \partial H(u_0)$.
\end{remark}

\begin{lemma}\label{diff}
Given $u\in BV(\Omega)$ and $\z \in L^\infty(\Omega;\RR^N)$ with $\|\z\|_\infty \leq 1$, $\d\z\in L^N(\Omega)$, $\left(\z,Du  \right)=|Du|$ and $\left[\z,\nu  \right]\in \rm{sign}(-u)$ on $\partial \Omega$. Let $\xi_u\>:\>BV(\Omega)\to \RR$ be a linear map defined as
\[
\xi_u(v):=-\io  v\,\d \z \,.
\]
Then, $\xi_u \in \partial \|u\|$.
\end{lemma}
  \begin{proof}
  Observe that $\xi_u \in BV(\Omega)^\prime$ as a consequence of the Anzellotti theory. Indeed, Green's formula \eqref{Green-2} and $\|\z\|_\infty\le1$ imply
  \begin{multline*}
  |\xi_u(v)|\le\left|\io (\z, Dv)\right|+\left|\ido v \left[\z,\nu  \right]d\mathcal{H}^{N-1}\right|\\
  \le\io |Dv|+\ido |v|\, d\mathcal{H}^{N-1}\,,
  \end{multline*}
  for every $v\in BV(\Omega)$.
  So $\xi_u \in BV(\Omega)^\prime$ and $\|\xi_u\|\le 1$.

  On the other hand, for every $v\in BV(\Omega)$ we obtain
  \begin{align*}
  \xi_u\,(v-u) &= \io - \d \, \z \,(v-u)
  \\
  & =\io \left(\z, D(v-u)   \right)-\ido (v-u)\, \left[\z,\nu  \right]d\mathcal{H}^{N-1}
  \\
  & = \io \left(\z, Dv   \right)-\io |Du|-\ido (v\, \left[\z,\nu  \right]+|u|)d\mathcal{H}^{N-1}
  \\
  & \leq \|\z\|_\infty \io |Dv|-\io |Du|+\|\z\|_\infty \ido |v|d\mathcal{H}^{N-1}- \ido |u|d\mathcal{H}^{N-1}
  \\
  & \leq \|v\|-\|u\|.
  \end{align*}

  \end{proof}

Let $J:BV(\Omega)\to \RR$ be defined as
\[
J(u)=\io |Du|+\ido |u|\,d\mathcal{H}^{N-1}-\io F(x,u).
\]
  We will say that $u_0\in BV(\Omega)$ is a critical point of functional $J$ if there exists $\z \in L^\infty(\Omega;\RR^N)$ with $\|\z\|_\infty \leq 1$ such that
  \begin{align*}
  -\io w\, \d\z=\io f(x,u_0)w,\qquad \hbox{ for all }w \in BV(\Omega),
  \\
  \left(\z, Du_0  \right)= |Du_0| \, \hbox{ in }\Omega \quad \hbox{ and }\quad \left[\z,\nu  \right]\in \rm{sign}(-u_0) \, \hbox{ on }\partial \Omega.
  \end{align*}

 In virtue of Lemma \ref{diff}, the functional given by $\xi(w)=-\io w\, \d\z$ belongs to $\partial\|u_0\|$. We point out that critical points of $J$ coincide with solutions of problem \eqref{problem}.
\medskip

\section{Proof of Theorem 1}

\subsection{Existence of non trivial solutions}

We shall prove that \eqref{problem} has a nontrivial solution $w \geq 0$. A similar argument shows that there exists a nontrivial solution $v\leq 0$.

  Let $\tilde p = \min \left \{1+\alpha, \kappa, q+1   \right \}$. For each $1<p<\tilde p$, consider the problem
\begin{equation}\label{p-laplacian}
\left\{
\begin{array}{cc}
-\textrm{div}\left(|\nabla u|^{p-2}\nabla u \right)=f(x,u),& \hbox{ in } \Omega,
\\
u=0 & \hbox{ on } \partial \Omega.
\end{array}
\right.
\end{equation}
By our hypotheses and the choice of $\tilde p$, the following assertions are true for every $p\in (1,\tilde p)$:
\begin{enumerate}
\item[(a)] $|f(x,s)|\leq C(1+|s|^q)$ with $0<q<p^*-1,$
\item[(b)] $\lim_{s\to 0} \sup \displaystyle \frac{f(x,s)}{|s|^{p-2}s}=0$, uniformly with $x\in \Omega$,
\item[(c)] $0<\kappa F(x,s)\leq sf(x,s)$ for $x\in \Omega$, $|s|\geq s_0$ and $\kappa >p$.
\end{enumerate}
Then, it is well--know that problem \eqref{p-laplacian} has nontrivial solutions $v_p\leq 0 \leq w_p$ (see e.g. \cite{Mawhin}).
These solutions are obtained using the ``Mountain Pass Theorem" by Ambrosetti and Rabinowitz (\cite{AR}) for the two following functionals
$J_p^{\pm}:W_0^{1,p}(\Omega)\to \RR$ given by
\[
J_p^{\pm}(u)=\frac{1}{p}\io |\nabla u|^p-\io F_{\pm}(x,u),
\]
where $F_{\pm}(x,s)=\int_0^s f_{\pm}(x,t){\rm d}t$, being $f_{\pm}:\Omega \times \RR \to \RR$ defined by
\begin{equation*}
f_+(x,s)=
\left \{
\begin{array}{lc}
0 &\, \hbox{ if } s\leq 0,
\\
f(x,s) &\, \hbox{ if } s>0.
\end{array}
\right.
\quad
f_-(x,s)=
\left \{
\begin{array}{lc}
f(x,s) &\, \hbox{ if } s\leq 0,
\\
0 &\, \hbox{ if } s>0.
\end{array}
\right.
\end{equation*}
  Concretely, for the nonnegative solution $w_p$ it is used $J^+_p$ (while $J_p^-$ is used for the nonpositive one $v_p$). Now consider the functional \[I_p(u)=J_p^+(u)+\frac{p-1}{p}\left| \Omega  \right|.\] Since, by Young's inequality
  \[
  \io |\nabla u|^{p_1}\leq \frac{p_1}{p_2}\io |\nabla u|^{p_2}+\frac{p_2-p_1}{p_2}|\Omega|,  \quad 1\leq p_1\leq p_2,
  \]
  it follows that $I_p$ is nondecreasing with respect to $p$. On the other hand, we fix $0<\phi \in \mathcal{C}_c^\infty(\Omega)$ and since $I_p(t\phi)\to -\infty$ as $t \to \infty$, it yields $e=T\phi$  (for some $T>0$)  such that $I_{\tilde p}(e)<0$. Then, by monotonicity, we obtain
  \[
  I_p(e)<0, \quad \hbox{ for all } p\in (1,\tilde p).
  \]

   Moreover, due to the fact that critical points of $J_p^+$ are uniquely determined by critical points of $I_p$, it follows that $u\equiv 0$ is a local minimum of $I_p$ and $w_p\geq 0$ is a nontrivial critical point of $I_p$ which can be obtained invoking to the Mountain Pass Theorem. That is, it satisfies
  \[
  I_p(w_p)=\inf_{\gamma \in \Gamma_p} \max_{t\in [0,1]} I_p(\gamma (t)),
  \]
where
\[
\Gamma_p =\left \{ \gamma \in \mathcal{C}\left([0,1],W_0^{1,p}(\Omega)   \right) \,:\, \gamma(0)=0, \gamma(1)=e  \right\}.
\]
Next we claim that the sequence $\left\{ I_p(w_p) \right\}_{1<p<\tilde p}$ is increasing. Indeed, let $1<p_1<p_2<\tilde p$ and thanks to the monotony of $I_p$ and the fact that $\Gamma_{p_2}\subset \Gamma_{p_1}$ (because $W_0^{1,p_2}(\Omega)\subset W_0^{1,p_1}(\Omega)$), it holds
\begin{align*}
I_{p_1}(w_{p_1}) & =\inf_{\gamma \in \Gamma_{p_1}} \max_{t\in [0,1]} I_{p_1}(\gamma (t))
\\
&\leq \inf_{\gamma \in \Gamma_{p_2}} \max_{t\in [0,1]} I_{p_1}(\gamma (t))
\\
&\leq \inf_{\gamma \in \Gamma_{p_2}} \max_{t\in [0,1]} I_{p_2}(\gamma (t))
\\
&=I_{p_2}(w_{p_2})
\end{align*}
and the claim is proved. Thus, for a fixed $p_0\in (1,\tilde p)$ we get $I_p(w_p)\leq I_{p_0}(w_{p_0})$ for all $p\in (1,p_0)$ and hence
\begin{equation}\label{polo1}
\frac{1}{p}\io |\nabla w_p|^p-\io F(x,w_p) \leq  C, \quad \hbox{ for all } p\in (1,p_0),
\end{equation}
with $C=C(p_0)>0$  independent of $p$. Observe that we write $F(x,w_p)$ instead $F_+(x,w_p)$ because $w_p\geq 0$ (an analogous remark holds for $f_+(x,w_p)$).

  We denote $\Omega_p=\left \{x\in \Omega \, : \, w_p(x)\leq s_0    \right \}$, for any $p\in (1,p_0)$. Then, by  condition (a) and the definition of $F(x,s)$, we obtain
  \begin{equation}\label{F1}
  \int_{\Omega_p}F(x,w_p)\leq Cs_0\left(1+s_0^{q}   \right)|\Omega|=C_1,
  \end{equation}
  where $C_1$ is independent of $p$. Also, by condition (c) and since $w_p$ is a solution, it holds
   \begin{equation}\label{F2}
  \int_{\Omega \setminus \Omega_p}F(x,w_p)\leq \frac{1}{\kappa}\int_{\Omega} w_pf(x,w_p)= \frac{1}{\kappa} \io |\nabla w_p|^p.
  \end{equation}
Substituting \eqref{F1} and \eqref{F2} into \eqref{polo1}, we get
\[
\left(\frac{1}{p_0}-\frac{1}{\kappa}   \right)\io |\nabla w_p|^p\leq
\left(\frac{1}{p}-\frac{1}{\kappa}   \right)\io |\nabla w_p|^p\leq C+C_1\,.
\]
Then, since $\kappa > p_0$, we conclude that
\begin{equation}\label{bounded-p}
\io |\nabla w_p|^p\leq \tilde C,\qquad \forall \, p\in (1,p_0),
\end{equation}
for some positive constant $\tilde C=\tilde C(p_0)$, independent of $p$.


This last inequality \eqref{bounded-p} allows us to establish the following statements (see \cite[Proposition 3]{ABCM}, and also \cite[Theorem 3.3]{MRST}): there exists a bounded vector field  $\z \in L^\infty(\Omega:\RR^N)$ with $\|\z\|_\infty \leq 1$ such that
\begin{equation}
\label{convergence-z}
|\nabla w_p|^{p-2}\nabla w_p \rightharpoonup \z, \, \hbox{ weakly in } L^r(\Omega;\RR^N), \, \hbox{ for all } 1\leq r <\infty,
\end{equation}
as $p\to 1^+$. In particular,
\begin{equation}\label{convergence-gradient-term}
\io |\nabla w_p|^{p-2}\nabla w_p\cdot \nabla \varphi \to \io \z \cdot \nabla \varphi,\, \hbox{ for all } \varphi \in \mathcal{C}_c^1(\Omega).
\end{equation}

On the other hand, \eqref{bounded-p} and Young's inequality imply
$$\|w_p\|\le \int_{\partial\Omega}|w_p|\, d\mathcal H^{N-1}+\frac1p \io |\nabla w_p|^p+\frac{p-1}p|\Omega|\leq \tilde C+|\Omega|\,,$$
so that $\{w_p\}_{p>1}$ is bounded in $BV(\Omega)$.
 It follows that there exists $w\in BV(\Omega)$ such that, up to a subsequence (no relabeled),
\begin{itemize}
\item[(A)] $w_p \to w, \hbox{ in } L^m(\Omega), \, \hbox{ for } 1\leq m <\frac{N}{N-1}$.
\item[(B)] $w_p(x) \to w(x)$, almost everywhere $x\in \Omega$.
\item[(C)] $\exists \, g\in L^m(\Omega)$  ($1\leq m<\frac{N}{N-1}$) such that $|w_p(x)|\leq g(x)$.
\end{itemize}
Observe that $w\geq 0$ because $w_p\geq 0$ for all $p>1$. Then, thanks to (B) and the fact that $f(x,s)$ is a Carath\'eodory function, we obtain
\[
f(x,w_p(x))\to f(x,w(x)), \qquad \hbox{ a.e. } x\in \Omega.
\]
Moreover, we deduce from  (C) that
\[
|f(x,w_p(x))|\leq C(1+|w_p(x)|^q)\leq C(1+g(x)^q)\in L^{N}(\Omega).
\]
Consequently, by the Dominated Convergence Theorem,
\begin{equation}\label{converge-f}
\io f(x,w_p) \varphi \to \io f(x,w)\varphi, \, \qquad  \hbox{ for all } \varphi \in \mathcal{C}_c^1(\Omega).
\end{equation}
Expressions \eqref{convergence-gradient-term} and \eqref{converge-f} imply that
\begin{equation}\label{proof1}
-\d \, \z =f(x,w) \hbox{ in } \mathcal{D}^\prime (\Omega).
\end{equation}
In order to prove that $\left(\z, Dw\right)=|Dw|$, we note that it is enough to show $\left<({\bf z}, Dw), \varphi \right> = \left< |Dw|, \varphi \right>$ for all $0\leq \varphi \in \mathcal{C}_c^1(\Omega)$. Since $\|\z\|_\infty \leq 1$ and \eqref{Borel} holds, we just prove the inequality $\left<({\bf z}, Dw), \varphi \right> \ge \left< |Dw|, \varphi \right>$. Due to the definition of $(\z, Dw)$, we must check that:
\begin{equation}\label{second}
-\io w \, \d \, \z \, \varphi - \io w\, \z \cdot \nabla \varphi \geq \io |Dw|\, \varphi, \, \hbox{ for all } 0\leq \varphi \in \mathcal{C}_c^1(\Omega).
\end{equation}
To this end, taking $0\leq w_p\, \varphi \in W_0^{1,p}(\Omega)$ as a test function in problem \eqref{p-laplacian}, we get
\begin{equation}\label{letting}
\io |\nabla w_p|^p \varphi +\io w_p |\nabla w_p|^{p-2}\nabla w_p\cdot \nabla \varphi = \io f(x,w_p)w_p\, \varphi.
\end{equation}
 We estimate the first integral term in \eqref{letting} using Young's inequality:
 \[
 \io \varphi|\nabla w_p|\le\frac1p \io\varphi|\nabla w_p|^p+\frac{p-1}p\io \varphi\,.
 \]
 Now, from the lower semicontinuity of the involved functional, we obtain
\begin{align*}
\liminf_{p\to 1^+}  \io \varphi |\nabla w_p|^p   &  \geq \liminf_{p\to 1^+}  \io \varphi |\nabla w_p|
 \\
 & =\io \varphi |Dw|\,.
\end{align*}
On the other hand, by (A) and \eqref{convergence-z}
\[
\io w_p |\nabla w_p|^{p-2}\nabla w_p\cdot \nabla \varphi \to \io w \, \z \cdot \nabla\varphi, \quad \hbox{ as } p\to 1^+.
\]

The right hand side of \eqref{letting} is analyzed as follows.  We deduce from
\[
|f(x,w_p)w_p\, \varphi |\leq M C |w_p|(1+|w_p|^q)\leq C_1g(x)(1+g(x)^q)\in L^1(\Omega)
\]
 and the pointwise convergence, that
\[
\io f(x,w_p)w_p\, \varphi \to \io f(x,w)w\, \varphi = -\io \d \, \z \, w \, \varphi.
\]
Then, letting $p \to 1^+$ in \eqref{letting}, we obtain the required inequality \eqref{second} to conclude that
\begin{equation}\label{proof2}
\left(\z, Dw\right)=|Dw|.
\end{equation}
Next, we will show that $\left[\z, \nu   \right ]\in {\rm{sign}}(-w)$ on $\partial \Omega$. It is easy to check that this fact  is equivalent to show
\begin{equation}\label{sergio}
\ido \left(|w|+w \left[\z, \nu   \right ]  \right)d\mathcal{H}^{N-1}=0,
\end{equation}
because $|\left[\z, \nu   \right ]|\leq \| \z \|_\infty \leq 1$. Since
$
-w \left[\z, \nu   \right ] \leq \| \z \|_\infty  |w| \leq  |w|
$
and so
\[
\ido \left(|w|+w \left[\z, \nu   \right ]  \right)d\mathcal{H}^{N-1}\ge 0\,,
\]
it remains to prove the reverse inequality. To do this, we take $w_p-\varphi$, with $\varphi \in \mathcal{C}_c^1(\Omega)$, as a test function in \eqref{p-laplacian}, to obtain
\begin{equation}\label{sergio2}
\io |\nabla w_p|^p =\io |\nabla w_p|^{p-2}\nabla w_p\cdot \nabla \varphi+\io f(x,w_p)(w_p-\varphi).
\end{equation}
Hence, using Young's inequality, we get
\begin{multline*}
p\io |\nabla w_p|  \leq \io |\nabla w_p|^p + (p-1)|\Omega|
\\
 =\io |\nabla w_p|^{p-2}\nabla w_p\cdot \nabla \varphi+ \io f(x,w_p)(w_p-\varphi)+ (p-1)|\Omega|.
\end{multline*}
Now, having in mind \eqref{convergence-z}, the weak lower semicontinuity of the total variation and from the previous arguments, we can pass to the limit as $p\to 1^+$, to have
\begin{align}\label{sergio3}
\nonumber \io |Dw| + \ido |w|d\mathcal{H}^{N-1} & \leq \io \z \cdot \nabla \varphi -\io f(x,w)\varphi+\io f(x,w)w
\\
 &= \io f(x,w)w,
\end{align}
due to \eqref{proof1}. Furthermore, by  \eqref{proof1},
 \eqref{Green-2} and \eqref{proof2}, we get
\begin{align*}
\io f(x,w)w & =-\io w\, \d \, \z
\\
& =-\ido w \left[\z,\nu  \right]d\mathcal{H}^{N-1}+\io \left(\z,Dw   \right)
\\
&=-\ido w \left[\z,\nu  \right]d\mathcal{H}^{N-1} + \io |Dw|\,.
\end{align*}
Replacing this equality in \eqref{sergio3} gives the desired equality in \eqref{sergio} and we conclude that
\begin{equation}\label{proof3}
\left[\z, \nu   \right ]\in {\rm{sign}}(-w) \hbox{ on } \partial \Omega.
\end{equation}
Then, \eqref{proof1}, \eqref{proof2} and \eqref{proof3} lead to conclude that $w$ is a nonnegative solution of problem \eqref{problem} in the sense of Definition \ref{definition}.

  In order to check that $w$ is nontrivial, by hypothesis $(i)$, $f(x,0)=0$ and there exists $\delta>0$, small enough, such that $|f(x,s)|\leq K_1|s|^\alpha$ for all $|s|\in (0,\delta)$ and for some $K_1>0$. Observe that hypothesis $(ii)$ implies $\alpha <q<\frac{1}{N-1}$. Moreover, by definition of $F_+(x,s)$ it follows
  \[
  F_+(x,s)=\int_0^sf_+(x,t)dt\leq \int_0^s |f(x,s)|\leq \frac{K_1}{1+\alpha}|s|^{1+\alpha},
  \]
for $|s|\in (0,\delta)$. Let $\rho \in (0,\delta)$ to be determined. Then, for $u\in BV(\Omega)$ with $\|u\|=\rho$, it holds
\begin{align*}
J(u) &= \|u\|-\io F_+(x,u)
\\
& \geq \|u\|-\frac{K_1}{1+\alpha}\io |u|^{1+\alpha}
\\
& \geq \|u\|-K_2\|u\|^{1+\alpha}
\\
& = \rho (1-K_2\rho^\alpha).
\end{align*}
We define $\rho$, so small, such that $1-K_2\rho^\alpha\geq \frac{1}{2}$, so that
\[
J(u)\geq \frac{\rho}{2},\quad \hbox{ for }\|u\|=\rho>0.
\]
Observing that $J(e)<0$, we deduce that $\|e\|>\rho$.
 Since, by Young's inequality, we get that $I_p(u)\geq J(u)$ for all $u\in W_0^{1,p}$, it follows that
 \begin{equation}\label{tranvia}
  I_p(w_p)=\inf_{\gamma \in \Gamma_p} \max_{t\in [0,1]} I_p(\gamma (t))\geq \frac{\rho}{2}.
  \end{equation}
On the other hand, we have
\begin{align*}
\lim_{p\to 1^+}\frac{1}{p}\io |\nabla w_p|^p &=\lim_{p\to 1^+}\frac{1}{p}\io f(x,w_p)w_p
\\
&=\io f(x,w)w
\\
& = \io (\z, Dw)-\ido w[\z, \nu]\, d\mathcal H^{N-1}
\\
& = \io |Dw|+\ido |w|\,d\mathcal{H}^{N-1},
\end{align*}
where in the last equality we have used that $w$ is a solution of \eqref{problem}. In addition, it is easy to check that
\[
\lim_{p\to 1^+} \io F(x,w_p)=\io F(x,w).
\]
By using these last two equalities, we can assert that
\begin{equation}\label{tiempo}
\lim_{p\to 1^+}I_p(w_p)=J(w).
\end{equation}
Summarizing \eqref{tranvia} and \eqref{tiempo} we conclude that $J(w)\geq \frac{\rho}{2}$ and then $w$ is nontrivial, because $J(0)=0$.
\medskip

With regard to the existence of a nontrivial solution $v\leq 0$ of problem \eqref{problem}, we use the same reasoning applied to the functional
\[
\tilde I_p(u)=\frac{1}{p}\io |\nabla u|^p-\io F_{-}(x,u)+\frac{p-1}{p}\left| \Omega  \right|,
\]
getting that $v_p \to v$ as $p\to 1^+$. Where $v_p$ is the nonpositive solution of $p-$Laplacian problem \eqref{p-laplacian}.


\subsection{Boundedness of the solutions}
In this subsection, we will write $\mathcal{S}_1$ to denote the best constant of the Sobolev embedding $W_0^{1,1}(\Omega)\hookrightarrow L^{\frac{N}{N-1}}(\Omega)$.
Moreover, for every $k\geq 0$ and $0\leq w_p \in W_0^{1,p}(\Omega) $  solution of \eqref{p-laplacian} defined in the proof of Theorem \ref{teo1}, we set
\[
A_k(w_p)=A_{k, p}=\left\{x\in \Omega \,:\, |w_p(x)|>k   \right\}.
\]

\begin{lemma}\label{lemilla}For every $\varepsilon >0$ there exists $k_0>0$ (which does not depend on $p$) such that
\[
\int_{A_{k, p}}(1+w_p^q)^N < \varepsilon
\]
for every $k\geq k_0$ and for all $p>1$ small enough.
\end{lemma}

\begin{proof}
Using H\"older's inequality twice, Sobolev's inequality  and taking into account that $$|A_{k, p}|\leq \frac{1}{k^{\frac{N}{N-1}}}\int_{A_{k, p}}w_p^{\frac{N}{N-1}},$$
we obtain
\begin{align*}
\int_{A_{k, p}}(1+w_p^q)^N & \leq 2^{N-1}\left( |A_{k, p}|+\int_{A_{k, p}}w_p^{qN}  \right)
\\
& \leq 2^{N-1}\left(|A_{k, p}|+\left(\int_{A_{k, p}}w_p^{\frac{N}{N-1}}   \right)^{q(N-1)}|A_{k, p}|^{1-q(N-1)}           \right)
\\
& \leq \frac{2^{N-1}(1+k^{qN})}{k^{\frac{N}{N-1}}}\int_{\Omega}w_p^{\frac{N}{N-1}}
\\
& \leq \frac{2^{N-1}(1+k^{qN})}{k^{\frac{N}{N-1}}}\mathcal{S}_1^{\frac{N}{N-1}}\left(\io |\nabla w_p|   \right)^{\frac{N}{N-1}}
\\
& \leq  \frac{2^{N-1}(1+k^{qN})}{k^{\frac{N}{N-1}}}\mathcal{S}_1^{\frac{N}{N-1}}\left(\io |\nabla w_p|^p   \right)^{\frac{N}{p(N-1)}}|\Omega|^{\frac{p-1}{p}\frac{N}{N-1}},
\end{align*}
now, having in mind inequality \eqref{bounded-p} which asserts the existence of a positive constant $\tilde C$, which does not depend on $p$, satisfying
\[
\left(\io |\nabla w_p|^p   \right)^{\frac{1}{p}}\leq \tilde C^{1/p}< 1+\tilde C
\]
and since $|\Omega|^{\frac{p-1}{p}}<1+|\Omega|$, it follows that there exists a positive constant $C=C(N,q,\mathcal{S}_1,|\Omega|)$ such that
\[
\int_{A_{k, p}}(1+w_p^q)^N < \frac{C(1+k^{qN})}{k^{\frac{N}{N-1}}}\to 0
\]
as $k \to \infty$, because $q<\frac{1}{N-1}$.
\end{proof}
\begin{remark}By a similar argument we can state the existence of a $k_0>0$ (which does not depend on $p$) such that
\[
\int_{A_{k, p}}(1+|v_p|^q)^N < \varepsilon
\]
for every $k\geq k_0$ and for all $p>1$ sufficiently small. Where $0\geq v_p\in W_0^{1,p}(\Omega)$ is the negative solution of \eqref{p-laplacian} and $A_{k, p}=A_k(v_p)$.
\end{remark}
Now, we are ready to prove the boundedness of the solutions $v$ and $w$ of problem \eqref{problem}.
\begin{proof}[Proof of Boundedness]
We prove the boundedness of the positive solution $w$. The proof for the negative one is similar in spirit.

 For every $k>0$, we define the auxiliary function $G_k:\RR \to \RR$ as usual
 \begin{equation*}
 G_k(s)= \left\{
 \begin{array}{lc}
 s-k, & s>k,
 \\
 0, & |s|\leq k,
 \\
 s+k,& s<-k.
 \end{array}
 \right.
 \end{equation*}
 Then, choosing $G_k(w_p)$ as a test function in \eqref{p-laplacian}, we get
 \begin{equation}\label{moll1}
 \io |\nabla G_k(w_p)|^p=\io f(x,w_p)G_k(w_p).
 \end{equation}
 Now, computing and using \eqref{moll1}, Sobolev's embedding, and the Young and H\"older inequalities, we have
 \begin{align*}
 \left(\io G_k(w_p)^{\frac{N}{N-1}}   \right)^\frac{N-1}{N} & \leq \mathcal{S}_1 \io |\nabla G_k(w_p)|
 \\
 &\leq \frac{\mathcal{S}_1}{p}\io |\nabla G_k(w_p)|^p+\frac{\mathcal{S}_1(p-1)}{p}|\Omega|
 \\
 &\leq \mathcal{S}_1 \io |f(x,w_p|G_k(w_p)+\frac{\mathcal{S}_1(p-1)}{p}|\Omega|
 \\
 & \leq  C \mathcal{S}_1 \int_{A_k}(1+w_p^q)G_k(w_p)+\frac{\mathcal{S}_1(p-1)}{p}|\Omega|
 \\
 & \leq   C \mathcal{S}_1 \left(\int_{A_{k, p}}(1+w_p^q)^N  \right)^{\frac{1}{N}}\left(\int_{\Omega}G_k(w_p)^{\frac{N}{N-1}}   \right)^{\frac{N-1}{N}}
 \\
 &\hspace{6,3cm} +\frac{\mathcal{S}_1(p-1)}{p}|\Omega|.
 \end{align*}
 By Lemma \ref{lemilla}, there exists $\tilde k_0>0$ (which does not depend on $p$) such that
 \[
 \int_{A_{k, p}}(1+w_p^q)^N<\frac{1}{(2C\mathcal{S}_1)^N}, \hbox{ for all } k\geq \tilde k_0,
 \]
 and for all $p>1$ sufficiently small. Consequently, we obtain
 \[
 \int_{\Omega}G_k(w_p)^{\frac{N}{N-1}}\leq \left(\frac{2\mathcal{S}_1(p-1)|\Omega|}{p}\right)^{\frac{N}{N-1}}.
 \]
 Since $w_p(x)\to w(x)$ a.e. $x\in \Omega$, by Fatou lemma, we can pass to the limit on $p \to 1$, to conclude that
 \[
 \int_{\Omega}(w(x)-k)^{\frac{N}{N-1}}=0, \hbox{ for every } k\geq \tilde k_0.
 \]
 Thus, $\|w\|_\infty \leq \tilde k_0$.
\end{proof}


\section{A Poho\u{z}aev type identity and explicit examples}
In this section we provide a Poho\u{z}aev type identity for elliptic problems involving the $1$--Laplacian operator
\begin{equation}\label{sense}
\left\{
\begin{array}{cc}
-\textrm{div}\left(\displaystyle \frac{Du}{|Du|} \right)=f(u),& \hbox{ in } \Omega,
\\
\\
u=0 & \hbox{ on } \partial \Omega.
\end{array}
\right.
\end{equation}
From now on, for any function $g$ evaluated on $\partial \Omega$, we write $\ido g$ instead of $\ido g \,d\mathcal{H}^{N-1}$ when no confusion can arise.
\begin{proposition}{{\bf{[Poho\u{z}aev type identity for the $1$--Laplacian]}}}\label{Pohozaev} Let $u\in W^{1,1}(\Omega)$ be a solution of problem \eqref{sense}  in the sense of Definition \ref{definition} with $\z \in \mathcal{C}^1(\overline \Omega_\delta)$ (for some $\delta>0$ sufficiently small) and assume that $x \cdot \nabla u\in W^{1,1}(\Omega)$. Then, $u$ satisfies the identity
\begin{align}\label{pacosaez}
(N-1)\io uf(u)-N\io F(u)&+\ido F(u) x\cdot \nu
\\
\nonumber =\ido |\nabla u| x\cdot \nu -\ido \left(x \cdot \nabla u\right)&\left(\z \cdot \nu \right)+(N-1)\ido |u| \,.
\end{align}
\end{proposition}
\begin{proof}

By our assumption $x \cdot \nabla u\in W^{1,1}(\Omega)$, we have
\[
\nabla \left(x \cdot \nabla u   \right) = \nabla u + D^2u\cdot x,
\]
where $\left(D^2u\cdot x\right)_j =\sum_{i=1}^N \frac{\partial^2u}{\partial x_i \partial x_j}x_i$ ($j=1,\ldots , N$)
 belong to $L^1(\Omega)$. Moreover, by Stampacchia's Theorem, $\nabla \left(x \cdot \nabla u   \right)=0$ a.e. in the set $\{ x \cdot \nabla u=0 \}$ which implies
 \[
 \left(D^2u\cdot x\right)_j =0, \quad \hbox{ a.e. in } \{  |\nabla u|=0 \}.
 \]


Hence, integrating by parts and taking into account \eqref{strip},  we obtain
\begin{multline}\label{uno}
\io \d \, \z\, (x\cdot \nabla u)= \ido \left(x \cdot \nabla u\right)\left(\z \cdot \nu \right)-\io \z  \cdot \, \nabla(x\cdot \nabla u)
\\
=\ido \left(x \cdot \nabla u\right)\left(\z \cdot \nu \right)-\io |\nabla u|-\io \left(D^2u\cdot x  \right)\cdot \z\, ,
\end{multline}
 On the other hand,  we also get
\begin{multline}\label{dos}
N\io |\nabla u|=\ido |\nabla u | x\cdot \nu- \io x\cdot \nabla (|\nabla u|)
\\
=\ido |\nabla u | x\cdot \nu-\io \left(D^2u\cdot x  \right)\cdot \z \,,
\end{multline}
where in the last integral term we replace $\frac{\nabla u}{|\nabla u|}$  by $\z$ since we can assume that $|\nabla u|>0$.
Then, combining \eqref{uno} and \eqref{dos}, we obtain
\begin{multline}\label{poho1}
\io \d \, \z\, (x\cdot \nabla u)= \\
\ido \left(x \cdot \nabla u\right)\left(\z \cdot \nu \right)+(N-1)\io |\nabla u|-\ido |\nabla u | x\cdot \nu.
\end{multline}
Since $u$ is a solution, we can choose $x\cdot \nabla u\in W^{1,1}(\Omega)$ as a test function and by using integration by parts we get
\begin{align*}
\io \d \, \z\, (x\cdot \nabla u) &=-\io f(u)(x \cdot \nabla u)
\\
&=-\sum_i \io x_i \frac{\partial F(u)}{\partial x_i}
\\
& =-\ido F(u) x\cdot \nu \,+N\io F(u).
\end{align*}
Also, taking $u$ as a test function we have
\begin{align*}
\io |\nabla u|=\io uf(u)+\ido u\left(\z \cdot \nu  \right)\,.
\end{align*}
Replacing the above two equalities in \eqref{poho1} and remembering that $u\left(\z \cdot \nu  \right)=-|u|$, it yields the equality \eqref{pacosaez}. Finally, we point out that in case $|\nabla u|=0$ in the whole $\Omega$, we obtain the identity
\[
(N-1)\io uf(u)-N\io F(u)+\ido F(u) x\cdot \nu =(N-1)\ido |u|
\]

\end{proof}

\begin{corollary}In case $\Omega=B_R$ (the ball of radius $R>0$). Under the hypotheses of Proposition \ref{Pohozaev},  solutions of \eqref{sense} must satisfy the inequality
\[
(N-1)\int_{B_R}uf(u)-N\int_{B_R}F(u)+R\int_{\partial B_R}F(u) \geq (N-1)\int_{\partial B_R} |u|.
\]
\end{corollary}
\begin{proof}
Since $x\cdot \nu =R$ and by \eqref{Borel} , it follows that
\begin{multline*}
\int_{\partial B_R} |\nabla u| x\cdot \nu \,-\int_{\partial B_R} \left(x \cdot\nabla u\right)\left(\z \cdot \nu \right) \\
 \geq R\int_{\partial B_R} |\nabla u|\,-\|\z\|_{\infty}\int_{\partial B_R} \left(x \cdot \nabla u\right)
 \\
  \geq R\left(1- \|\z\|_{\infty} \right)  \int_{\partial B_R} |\nabla u|\geq 0.
\end{multline*}
Substituting into \eqref{pacosaez}, we obtain the desired inequality.
\end{proof}

 The following result, first obtained by F. Demengel in \cite[Section 4]{D1}, is now a consequence of Proposition \ref{Pohozaev}.
\begin{corollary}
Besides the hypotheses of Proposition \ref{Pohozaev}, assume  that $u_{|\partial \Omega}\equiv 0$. Then
\[
(N-1)\io uf(u)=N\io F(u).
\]
 In particular, for $f(s)=|s|^{q-1}s$ it follows $q=\frac{1}{N-1}$.
\end{corollary}

  It is worth noting that in the Poho\u{z}aev inequalities, there is no restriction on the possible values of $q$. We give some explicit examples about radial solutions of problem \eqref{problem} in the ball $B_R=\{x\in \RR^N\>:\>|x|<R\}$. We point out  that they also satisfy the Poho\u{z}aev identity \eqref{pacosaez}.

  \begin{example}
  For $f(s)=|s|^{q-1}s$, with $q>0$
\[
u(x)\equiv \left(\frac{N}{R} \right)^{1/q}, \quad \z(x)=-\frac{x}{R},
\]
defines a positive constant solution, while a negative solution is defined by
 \[
u(x)\equiv -\left(\frac{N}{R} \right)^{1/q}, \quad \z(x)=\frac{x}{R}.
\]
Furthermore thanks to Proposition \ref{Pohozaev}, for a general continuous and increasing function $f$, constant solutions of \eqref{sense} in $B_R$ must satisfy
\[
u\equiv f^{-1}\left(\frac{N}{R} \right).
\]
  \end{example}

  In the next examples, we assume a supercritical growth, so that in the supercritical case, two positive (and two negative) solutions are obtained. A further remark is in order. We have considered the Anzellotti theory of pairing gradients of $BV$--functions and bounded vector fields whose divergence is an $L^N$--function. It should be remarked that analogous results hold for bounded vector fields whose divergence is a function belonging to the Marcinkiewicz space $L^{N,\infty}(\Omega)$. This fact is a consequence of the continuous embedding of $BV(\Omega) \hookrightarrow L^{\frac N{N-1},1}(\Omega)$,
  where  $L^{\frac N{N-1},1}(\Omega)$ denotes the Lorentz space (see \cite{Al}). Hence, the Radon measure $(\z, Du)$ is well--defined for the vector field $\z(x)=\frac{x}{|x|}$, whose distributional divergence is given by $\d\z(x)=\frac{N-1}{|x|}$ and belongs to $L^{N,\infty}(B_R)$, and for any $u\in BV(B_R)$.

\begin{example}
\
\begin{enumerate}
\item  For $f(s)=s_+^q$ with $q>\frac{1}{N-1}$
\[
u(x)=\left(\frac{N-1}{|x|} \right)^{1/q}, \quad \z(x)=-\frac{x}{|x|},
\]
is a positive solution in $W^{1,1}(B_R)$.
\item For $f(s)=\left(\left(\frac{N-1}{R} \right)^{1/q}+s \right)_+^q$ with $q>\frac{1}{N-1}$
\[
u(x)=\left(\frac{N-1}{|x|} \right)^{1/q}-\left(\frac{N-1}{R} \right)^{1/q}, \quad \z(x)=-\frac{x}{|x|},
\]
is a positive solution belongs to $W_0^{1,1}(B_R)$.
\end{enumerate}
\end{example}

\section*{Ackonwledgement}
This work was carried out during a stay by the first author at Universitat de Val\`encia (Spain) partially supported by Secretar\'ia de Estado de Investigaci\'on, Desarrollo e Innovaci\'on EEBB2016 (Spain). He wants to thank for the very nice and
stimulating atmosphere found there.  The authors also want to thank Julio Rossi for inspiring discussions concerning the subject of this paper.


\begin{thebibliography}{99}



\bibitem{Al} A. Alvino, Sulla diseguaglianza di Sobolev in spazi di Lorentz.
  \emph{Boll. Un. Mat. Ital. A} (5) \textbf{14} (1977), no. 1, 148--156.

\bibitem{AA}  A. Ambrosetti and D. Arcoya, An introduction to nonlinear functional analysis and elliptic problems. \emph{Progress in Nonlinear Differential Equations and their Applications,} 82. Birkh\"auser Boston, Inc., Boston, MA, 2011. xii+199 pp.

\bibitem{AR} A. Ambrosetti and P.H. Rabinowitz, Dual variational methods in critical point theory and applications, \emph{Journal of Functional Analysis}, \textbf{14}, 349--381 (1973).

\bibitem{AFP} L. Ambrosio, N. Fusco and D. Pallara, Functions of bounded variation and free discontinuity problems. \emph{Oxford Mathematical Monographs}. The Clarendon Press, Oxford University Press, New York, 2000. xviii+434 pp. ISBN: 0--19--850245--1


\bibitem{ABCM}F. Andreu, C. Ballester, V. Caselles, and J. M. Maz\'on, The Dirichlet problem for the total variation flow. \emph{Journal of Functional Analysis}, \textbf{180}, 347--403 (2001).

    \bibitem{ACDM} F. Andreu, V. Caselles, J.I. D\'{\i}az and J.M.  Maz\'on,
Some qualitative properties for the total variation flow,
\emph{J. Funct. Anal.} {\bf  188}  (2002),  no. 2, 516--547.

\bibitem{ACM1} F. Andreu, V. Caselles and J.M. Maz\'on,
Existence and uniqueness of a solution for a parabolic quasilinear problem for linear growth functionals with $L^1$ data,
\emph{Math. Ann.} {\bf  322}  (2002),  no. 1, 139--206.

\bibitem{ACM} F. Andreu, V. Caselles, and J.M. Maz\'on, Parabolic Quasilinear Equations Minimizing Linear Growth Functionals. \emph{Progress in Mathematics}, vol. \textbf{223}, Birkhauser, (2004).

\bibitem{Anze} G. Anzellotti, Pairings Between Measures and Bounded Functions and Compensated Compactness. \emph{Ann. di Matematica Pura ed Appl.}, \textbf{ 135}  no. 1, 293--318 (1983).

    \bibitem{BCN}  G. Bellettini, V. Caselles and M. Novaga,
 The Total Variation Flow in $\RR^N$, \emph{J. Diff Equat.} {\bf
 184} (2002), 475--525.

 \bibitem{CT} M. Cicalese and C. Trombetti, Asymptotic behaviour of
 solutions to $p$--Laplacian equation, \emph{Asymptot. Anal.}
 {\bf 35} (2003), 27--40.


\bibitem{D1}  F. Demengel, On some nonlinear partial differential equations involving the 1--Laplacian and critical Sobolev exponent. \emph{ESAIM Control Optim. Calc. Var.} \textbf{4} (1999), 667--686.

\bibitem{D2}  F. Demengel, Th\'eor\`emes d'existence pour des \'equations avec l'op\'erateur ``1-Laplacien", premi\`ere valeur propre
pour $-\Delta_1$. \emph{C. R. Acad. Sci. Paris, Ser. I} \textbf{334}  (2002), 1071--1076.

  \bibitem{D3} F. Demengel, On some nonlinear equation involving the 1-Laplacian and trace map
  inequalities, \emph{Nonlinear Anal. T.M.A.} {\bf 48} (2002), 1151--1163.

\bibitem{DJM} G. Dinca, P. Jebelean and J. Mawhin, A result of Ambrosetti--Rabinowitz type for $p$--Laplacian. \emph{ Qualitative problems for differential equations and control theory}, 231--242, World Sci. Publ., River Edge, NJ, (1995).

\bibitem{Mawhin} G. Dinca, P. Jebelean and J. Mawhin, Variational and topological methods for Dirichlet problems
 with $p-$Laplacian. \emph{Portugaliae Mathematica}, vol. \textbf{58} Fasc. 3, 339--378 (2001).

 \bibitem{F} P.C. Fife, Mathematical aspects of reacting and diffusing systems. \emph{Lecture Notes in Biomathematics}, 28. Springer-Verlag, Berlin-New York, 1979. iv+185 pp.

 \bibitem{GP} N. Ghoussoub and D. Preiss, \emph{ A general mountain pass principle
for locating and classifying critical points}, Ann. Inst. H.
Poincar\'e Anal. Non Lin\'eaire \textbf{6} (1989), no. 5, 321--330.


    \bibitem{K}  B. Kawohl, \emph{ On a family of torsional creep problems,}
     J. Reine Angew. Math. \textbf{ 410} (1990), 1--22.

 \bibitem{K2} B. Kawohl,
From p-Laplace to mean curvature operator and related questions
 Progress in partial differential equations: the Metz surveys, 40--56, Pitman
 Res. Notes Math. Ser., 249, Longman Sci. Tech., Harlow, 1991.

\bibitem{MRST} A. Mercaldo, J.D. Rossi, S. Segura de Le\'on and C. Trombetti, Behaviour of $p-$Laplacian problems with Neumann boundary conditions when $p$ goes to $1$. \emph{Commun. Pure Appl. Anal.}, vol. \textbf{12}, no. 1, 253--267 (2013).

\bibitem{Ni} Wei-Ming Ni, The mathematics of diffusion. CBMS-NSF Regional Conference Series in Applied Mathematics, 82. \emph{Society for Industrial and Applied Mathematics (SIAM)}, Philadelphia, PA, 2011. xii+110 pp.



\end{thebibliography}
\end{document}